\documentclass[12pt,reqno]{amsart}
\usepackage{etoolbox}
\makeatletter
\patchcmd\maketitle
{\uppercasenonmath\shorttitle}
{}
{}{}
\patchcmd\maketitle
{\@nx\MakeUppercase{\the\toks@}}
{\the\toks@}
{}
{}{}
\patchcmd\@settitle{\uppercasenonmath\@title}{\Large}{}{}
\patchcmd\@setauthors
{\MakeUppercase{\authors}}
{\authors}
{}{}
\makeatother
\usepackage{amsmath,amssymb,amsthm,amsfonts}
\usepackage{mathtools}
\usepackage{color}
\usepackage{xcolor}
\usepackage{url}
\usepackage[shortlabels]{enumitem}
\usepackage{tikz}
\usepackage{pgfplots}
\pgfplotsset{compat=1.18}
\usepgfplotslibrary{fillbetween}
\usepackage[utf8]{inputenc}
\usepackage[T1]{fontenc}
\newtheorem{theorem}{Theorem}[section]
\newtheorem{corollary}[theorem]{Corollary}
\newtheorem{proposition}[theorem]{Proposition}
\newtheorem{lemma}[theorem]{Lemma}

\newtheorem{remark}[theorem]{Remark}
\newtheorem{example}[theorem]{Example}
\numberwithin{equation}{section}
\DeclareMathOperator{\Tr}{Tr}
\DeclareMathOperator{\Rea}{Re}
\usepackage[colorlinks=true]{hyperref}
\hypersetup{urlcolor=blue, citecolor=red, linkcolor=blue}
\usepackage{cite}
\usepackage{palatino,euler}
\begin{document}
	
	\title[Spectral Rigidity of Commutators]{Spectral Rigidity of Commutators:\\ Dynamics, Resonance, and Nilpotency}
	
	\keywords{commutator, inner derivation, nilpotency, Kleinecke--Shirokov theorem, similarity orbit}
	\subjclass[2020]{47B47, 47A10, 15A27, 34A30}
	
	\author[H.\ Stankovi\'c]{Hranislav Stankovi\'c}
	\address{Faculty of Electronic Engineering, University of Ni\v s, Aleksandra Medvedeva 4, 18000 Ni\v s, Serbia}
	\email{\url{hranislav.stankovic@elfak.ni.ac.rs}}
	
	\thanks{The results of this paper were presented at the International Congress of Mathematicians, Philadelphia, July 2026.}
	
	\date{\today}
	
	\maketitle
	
	\begin{abstract}
		Let $A, T \in M_n(\mathbb{C})$ and let $\Delta_A(T) = AT - TA$ denote the inner derivation induced by $A$. We determine when $T$ and $\Delta_A(T)$ are nilpotent under the second-order relation
		
		$$
		\Delta_A^2(T) + \alpha\, \Delta_A(T) + \beta\, T = 0, \qquad \alpha, \beta \in \mathbb{R},
		$$ 
		according to the location of the roots of $z^2 + \alpha z + \beta$. If the roots have nonzero real parts of the same sign, then both $T$ and $\Delta_A(T)$ are nilpotent. If the roots are purely imaginary and nonzero, no nilpotency conclusion holds in general, whereas if one root is zero and the other is nonzero, $\Delta_A(T)$ is nilpotent but $T$ need not be. The double zero root gives the Kleinecke--Shirokov theorem. For real roots of opposite signs, the answer is governed by an arithmetic threshold: writing $z_1/z_2=-p/q$ in lowest terms, every solution is nilpotent when $p+q>n$, while for $p+q\le n$ an explicit cyclic construction yields solutions for which both $T$ and $\Delta_A(T)$ are invertible. Finally, for a relation
		
		$$
		\sum_{k=0}^{m} c_k\,\Delta_A^k(T)=0
		$$ 
		of arbitrary order, every iterated commutator $\Delta_A^j(T)$ is nilpotent whenever the roots of the associated polynomial lie in an open half-plane whose boundary passes through the origin.
	\end{abstract}

	\bigskip
	
	\section{Introduction}\label{sec:intro}
	\bigskip
	Commutator identities of the form
	\[
	[A,[A,T]] = 0
	\]
	occupy a classical place in operator theory. The Kleinecke--Shirokov theorem \cite{kleinecke, shirokov} asserts that if $A$ and $T$ are elements of a complex Banach algebra satisfying this identity, then the commutator $[A,T]$ is quasi-nilpotent, which confirmed a conjecture of Kaplansky. In the finite-dimensional algebra $M_n(\mathbb{C})$, quasi-nilpotency coincides with nilpotency, so the conclusion becomes $[A,T]^n = 0$ for some $n\in\mathbb{N}$. The two original proofs are of a rather different nature. Writing $\Delta_A(T) = [A,T] = AT - TA$ for the inner derivation induced by $A$, Kleinecke's argument \cite{kleinecke} is combinatorial and rests on the identity $\Delta_A^n(T^n) = n!\, \Delta_A(T)^n$, valid whenever $\Delta_A^2(T) = 0$, which leads to the estimate
	\[
	\left\lVert \Delta_A(T)^n \right\rVert \leq \frac{(2 \left\lVert A \right\rVert\left\lVert T \right\rVert)^n}{n!}\, 
	\]
	and hence to $r(\Delta_A(T)) = 0$ via the Gelfand formula (see also \cite{halmos}). Shirokov's proof \cite{shirokov}, on the other hand, is analytic in spirit and already works with the conjugation $\lambda \mapsto e^{\lambda A}\, T\, e^{-\lambda A}$. In the finite-dimensional setting one may argue instead through Jacobson's lemma \cite{jacobson}. Variants of the theorem are still being studied, among them localizations \cite{BK}, a generalization for matrices \cite{KramarKramarFijavz23}, and a version for isometric transformations \cite{Sta}.
	
	The purpose of this paper is to take this dynamical point of view, present in Shirokov's argument in its simplest instance, and to develop it systematically. Rather than treating a polynomial relation $P(\Delta_A)(T) = 0$ solely as an algebraic constraint, we regard it as a differential equation along the similarity orbit
	\[
	F(\lambda) = e^{\lambda A}\, T\, e^{-\lambda A}, \qquad \lambda \in \mathbb{R}.
	\]
	The identity $F^{(k)}(\lambda) = e^{\lambda A}\, \Delta_A^k(T)\, e^{-\lambda A}$ (Lemma~\ref{lem:derivatives}) transforms such a relation into a constant-coefficient matrix differential equation, while every point of the trajectory remains similar to $T$, so the spectrum, the spectral radius, and the power traces are constant along the orbit. This creates a useful tension between analysis and algebra: the differential equation prescribes exponential, oscillatory, or polynomial growth of the orbit, whereas similarity preserves its spectral data. We use the term \emph{spectral squeezing} for the resulting mechanism: the orbit has constant spectral radius, while the differential equation forces its norm to decay along one unbounded temporal direction, so the constant radius must be zero. In other regimes, however, a nonzero spectrum may survive.
	
	More precisely, we investigate this phenomenon for the second-order relation
	\begin{equation}\label{eq:general_relation}
		\Delta_A^2(T) + \alpha\, \Delta_A(T) + \beta\, T = 0, \qquad \alpha, \beta \in \mathbb{R},
	\end{equation}
	where the location of the roots $z_1, z_2$ of the characteristic polynomial $P(z) = z^2 + \alpha z + \beta$ determines the dynamical regime. When both roots lie in the same open half-plane, the orbit decays exponentially in one temporal direction and both $T$ and $\Delta_A(T)$ are nilpotent. The degenerate configurations on the imaginary axis behave differently. For the Kleinecke--Shirokov relation $\Delta_A^2(T) = 0$ the commutator $\Delta_A(T)$ is nilpotent, whereas $T$ itself need not be, and for a nonzero purely imaginary pair of roots no nilpotency conclusion is possible in general. When one root vanishes and the other is real and nonzero, $\Delta_A(T)$ is nilpotent, while $T$ may even be invertible.
	
	The most delicate case is the hyperbolic regime $\beta < 0$, in which the roots are real and of opposite signs. The orbit then diverges in both temporal directions and spectral squeezing alone is insufficient. We show that nilpotency is in this case governed by the finite-dimensional resonance condition
	\[
	j z_1 + (k-j)\, z_2 \neq 0, \qquad 1 \leq j < k \leq n.
	\]
	Equivalently, writing $z_1/z_2 = -p/q$ in lowest terms ($p,q \in \mathbb{Z}^+$, $\gcd(p,q) = 1$), every solution of \eqref{eq:general_relation} is nilpotent when $p + q > n$, while for $p + q \leq n$ an explicit cyclic construction produces solutions that are even invertible.
	
	Finally, the one-sided squeezing argument extends to relations of arbitrary order, $\sum_{k=0}^{m} c_k\, \Delta_A^k(T) = 0$: if all roots of the associated polynomial lie in a single open half-plane, then every iterated commutator $\Delta_A^j(T)$ is nilpotent. From this viewpoint, the second-order results above treat in full the real-coefficient configurations in which the roots meet or straddle the imaginary axis, that is, the boundary cases of the half-plane principle at order two.
	
	\bigskip
	
	The paper is organized as follows. Section~\ref{sec:framework} sets up the differential framework and collects the analytic and algebraic tools. Sections~\ref{sec:imaginary}--\ref{sec:mixed} treat the cases according to the location of $z_1, z_2$ in the complex plane. Section~\ref{sec:resonance} analyzes the dimensional resonance specific to the hyperbolic case and supplies an explicit construction (Proposition~\ref{prop:construction}) realizing the non-nilpotent regime. Section~\ref{sec:higher} extends the method to higher-order relations. Section~\ref{sec:classification} collects the second-order results into a classification theorem with a phase-space reading, and Section~\ref{sec:conclusion} contains concluding remarks.
	
	\subsection*{Notation}
	Throughout, $M_n(\mathbb{C})$ denotes the algebra of $n \times n$ complex matrices with identity $I$. For $M \in M_n(\mathbb{C})$ we write $\sigma(M) \subseteq \mathbb{C}$ for its \emph{spectrum} (the set of eigenvalues) and
	\[
	r(M) := \max\{\,|\mu| : \mu \in \sigma(M)\,\}
	\]
	for its \emph{spectral radius}. The \emph{trace} of $M = [M_{ij}]_{i,j=1}^{n}$ is $\Tr(M) = \sum_{i=1}^{n} M_{ii}$, the sum of the eigenvalues of $M$ counted with multiplicity. A matrix $M$ is \emph{nilpotent} if $M^k = 0$ for some $k \geq 1$, equivalently (in $M_n(\mathbb{C})$) if $M^n = 0$, if $\sigma(M) = \{0\}$, or if $r(M) = 0$; these equivalences are used freely below. We write $\left\lVert\, \cdot\,\right\rVert$ for a submultiplicative matrix norm and $\Rea(z)$ for the real part of $z \in \mathbb{C}$. The \emph{inner derivation} induced by $A \in M_n(\mathbb{C})$ is
	\[
	\Delta_A(T) = [A,T] = AT - TA, \qquad T \in M_n(\mathbb{C}),
	\]
	and its iterates are defined by $\Delta_A^0(T) = T$ and $\Delta_A^k(T) = \Delta_A(\Delta_A^{k-1}(T))$ for $k \geq 1$. Since all norms on the finite-dimensional space $M_n(\mathbb{C})$ are equivalent, statements about convergence $X_k \to X$ and the induced convergence $r(X_k) \to r(X)$ do not depend on the choice of norm.
	
	\bigskip
	\section{The Matrix Differential Framework}\label{sec:framework}
	\bigskip
	
	Throughout this section we assume that $A, T \in M_n(\mathbb{C})$ satisfy the relation~\eqref{eq:general_relation} for fixed $\alpha, \beta \in \mathbb{R}$.

	\subsection{The similarity orbit and its derivatives}
	
	Define the matrix-valued trajectory $F: \mathbb{R} \to M_n(\mathbb{C})$ by the similarity (adjoint) action
	\begin{equation}\label{eq:similarity_orbit}
		F(\lambda) = e^{\lambda A}\, T\, e^{-\lambda A}.
	\end{equation}
	Since $e^{\lambda A}$ is invertible for all $\lambda \in \mathbb{R}$ with inverse $e^{-\lambda A}$, the matrix $F(\lambda)$ is similar to $T$ for every $\lambda$. Hence every similarity-invariant scalar is conserved along the orbit:
	\begin{equation}\label{eq:spectral_invariance}
		\sigma(F(\lambda)) = \sigma(T), \qquad r(F(\lambda)) = r(T), \qquad \Tr(F(\lambda)^k) = \Tr(T^k)
	\end{equation}
	for all $\lambda \in \mathbb{R}$ and all $k \geq 1$.
	
	\begin{lemma}\label{lem:derivatives}
		For each integer $k \geq 0$, the $k$-th derivative of $F$ with respect to $\lambda$ is
		\begin{equation}\label{eq:derivatives}
			F^{(k)}(\lambda) = e^{\lambda A}\, \Delta_A^k(T)\, e^{-\lambda A}.
		\end{equation}
	\end{lemma}
	\begin{proof}
		By induction on $k$. The case $k = 0$ is~\eqref{eq:similarity_orbit}. Assume the formula for some $k \geq 0$. Differentiating $F^{(k)}(\lambda) = e^{\lambda A}\, \Delta_A^k(T)\, e^{-\lambda A}$ and using $\frac{d}{d\lambda} e^{\lambda A} = A\, e^{\lambda A}$ and $\frac{d}{d\lambda} e^{-\lambda A} = -e^{-\lambda A}\, A$,
		\begin{align*}
			F^{(k+1)}(\lambda)&= A\, e^{\lambda A}\, \Delta_A^k(T)\, e^{-\lambda A} + e^{\lambda A}\, \Delta_A^k(T)\, e^{-\lambda A}\, (-A) \\
			&= e^{\lambda A}\bigl(A\, \Delta_A^k(T) - \Delta_A^k(T)\, A\bigr) e^{-\lambda A}
			= e^{\lambda A}\, \Delta_A^{k+1}(T)\, e^{-\lambda A},
		\end{align*}
		where we used that $A$ commutes with $e^{\lambda A}$.
	\end{proof}
	
	Under~\eqref{eq:general_relation}, substituting Lemma~\ref{lem:derivatives} shows that $F$ satisfies the second-order linear matrix ODE
	\begin{equation}\label{eq:matrix_ODE}
		F''(\lambda) + \alpha\, F'(\lambda) + \beta\, F(\lambda) = e^{\lambda A}\bigl(\Delta_A^2(T) + \alpha\,\Delta_A(T) + \beta\, T\bigr) e^{-\lambda A} = 0.
	\end{equation}
	
	\subsection{General solution and initial conditions}
	
	The form of the general solution of~\eqref{eq:matrix_ODE} depends on the roots $z_1, z_2$ of $P(z) = z^2 + \alpha z + \beta$, with discriminant $D = \alpha^2 - 4\beta$.
	
	\begin{enumerate}[\rm($i$)]
		\item \emph{Distinct roots} ($z_1 \neq z_2$): $F(\lambda) = e^{z_1 \lambda} C_1 + e^{z_2 \lambda} C_2$.
		\item \emph{Repeated root} ($z_1 = z_2 = z_0$): $F(\lambda) = e^{z_0 \lambda}(C_1 + \lambda\, C_2)$.
	\end{enumerate}
	The constants $C_1, C_2 \in M_n(\mathbb{C})$ are fixed by $F(0) = T$ and $F'(0) = \Delta_A(T)$.
	
	For distinct roots, solving $C_1 + C_2 = T$ and $z_1 C_1 + z_2 C_2 = \Delta_A(T)$ gives
	\begin{equation}\label{eq:C1C2_explicit}
		C_1 = \frac{\Delta_A(T) - z_2\, T}{z_1 - z_2}, \qquad C_2 = \frac{z_1\, T - \Delta_A(T)}{z_1 - z_2}.
	\end{equation}
	For the repeated root, $C_1 = T$ and $C_2 = \Delta_A(T) - z_0\, T$.
	
	\subsection{Analytic and algebraic tools}
	
	We use three standard properties of the spectral radius on $M_n(\mathbb{C})$.
	
	\begin{lemma}\label{lem:spectral_tools}
		Let $M, M_k \in M_n(\mathbb{C})$ and $\mu \in \mathbb{C}$. Then:
		\begin{enumerate}[\rm($i$)]
			\item \textup{(Continuity)} $r: M_n(\mathbb{C}) \to [0,\infty)$ is continuous; if $M_k \to M$ in any matrix norm, then $r(M_k) \to r(M)$.
			\item \textup{(Homogeneity)} $r(\mu M) = |\mu|\, r(M)$.
			\item \textup{(Norm bound)} $r(M) \leq \left\lVert M\right\rVert$ for any submultiplicative matrix norm $\left\lVert \cdot\right\rVert$.
		\end{enumerate}
	\end{lemma}
	
In $M_n(\mathbb{C})$, $r(M) = 0$ if and only if $M$ is nilpotent; thus $r(M)=0$ is equivalent to $M^n = 0$. To pass from traces to nilpotency we use the following classical consequence of Newton's identities, which we state separately since it is needed in the hyperbolic case  (where norm estimates are unavailable). We include a proof for the reader's convenience.
	
	\begin{lemma}\label{lem:newton}
		Let $X \in M_n(\mathbb{C})$. If $\Tr(X^k) = 0$ for all $1 \leq k \leq n$, then $X$ is nilpotent.
	\end{lemma}
	\begin{proof}
		Let $\mu_1,\dots,\mu_n$ be the eigenvalues of $X$ with multiplicity, let
		\[
		p_k = \sum_{i=1}^{n} \mu_i^k = \Tr(X^k), \qquad e_k = \sum_{1 \leq i_1 < \dots < i_k \leq n} \mu_{i_1} \cdots \mu_{i_k},
		\]
		be their power sums and elementary symmetric functions, with $e_0 = 1$. Newton's identities relate them by
		\[
		k\, e_k = \sum_{i=1}^{k} (-1)^{i-1}\, e_{k-i}\, p_i, \qquad 1 \leq k \leq n .
		\]
		By hypothesis $p_1 = \dots = p_n = 0$, so every term on the right-hand side contains a vanishing factor and $e_k = 0$ for all $1 \leq k \leq n$. Hence the characteristic polynomial of $X$ is $$\det(zI - X) = \sum_{k=0}^{n} (-1)^k e_k z^{n-k} = z^n,$$ and $X^n = 0$ by the Cayley--Hamilton theorem.
	\end{proof}
	
	To separate the traces from the exponential factors we use the following elementary independence property of exponentials (see, e.g., \cite{hsd}).
	
	\begin{lemma}\label{lem:exp_independence}
		Let $\mu_0, \dots, \mu_m$ be pairwise distinct complex numbers and let $c_0,\dots,c_m \in \mathbb{C}$. If $\sum_{j=0}^{m} c_j\, e^{\mu_j \lambda} = 0$ for all $\lambda \in \mathbb{R}$, then $c_j = 0$ for each $j$.
	\end{lemma}

	\bigskip
	\section{Roots on the Imaginary Axis}\label{sec:imaginary}
	\bigskip
	
	Roots on the imaginary axis occur in two regimes: the double zero root and pure imaginary pairs.
	
	\subsection{Double zero root ($z_1 = z_2 = 0$): the Kleinecke--Shirokov case}
	
	This is $\alpha = \beta = 0$, i.e.\ $\Delta_A^2(T) = 0$. The ODE reduces to $F''(\lambda) = 0$, with a linear trajectory
	\begin{equation}\label{eq:KS_trajectory}
		F(\lambda) = T + \lambda\, \Delta_A(T).
	\end{equation}
	
	The following short proof follows Shirokov's original idea.
	
	\begin{theorem}\label{thm:KS}
		Let $A, T \in M_n(\mathbb{C})$. If $\Delta_A^2(T) = 0$, then $\Delta_A(T)$ is nilpotent.
	\end{theorem}
	\begin{proof}
		From~\eqref{eq:KS_trajectory}, for $\lambda \neq 0$,
		\[
		\frac{F(\lambda)}{\lambda} = \frac{T}{\lambda} + \Delta_A(T) \xrightarrow[\lambda \to \infty]{} \Delta_A(T)
		\]
		in norm. By continuity and homogeneity of $r$ (Lemma~\ref{lem:spectral_tools}) and spectral invariance~\eqref{eq:spectral_invariance},
		\[
		r(\Delta_A(T)) = \lim_{\lambda \to \infty} r\!\left(\frac{F(\lambda)}{\lambda}\right) = \lim_{\lambda \to \infty} \frac{r(F(\lambda))}{|\lambda|} = \lim_{\lambda \to \infty} \frac{r(T)}{|\lambda|} = 0.
		\]
		Hence $\Delta_A(T)$ is nilpotent.
	\end{proof}
	
	\begin{remark}\label{rem:KS_geometry}
		This is the prototype of \emph{spectral squeezing}: the orbit traces an affine line along which the spectrum is constant while the norm grows linearly. A constant spectral radius together with unbounded norm forces that radius to be zero.
	\end{remark}
	
	\subsection{Conjugate imaginary roots ($z_{1,2} = \pm i\omega$, $\omega \neq 0$)}
	
	Here $\alpha = 0$, $\beta = \omega^2 > 0$, so $\Delta_A^2(T) + \omega^2 T = 0$. The system $F'' + \omega^2 F = 0$ has bounded oscillatory solutions
	\begin{equation}\label{eq:oscillatory}
		F(\lambda) = \cos(\omega\lambda)\, T + \frac{\sin(\omega\lambda)}{\omega}\, \Delta_A(T).
	\end{equation}
	Since $F(\lambda)$ is bounded and periodic of period $2\pi/\omega$, the constraint $r(F(\lambda)) = r(T)$ gives no asymptotic information, and the relation imposes no nilpotency on $T$ or $\Delta_A(T)$.
	
	\begin{example}\label{ex:oscillatory}
		Let $A = \begin{psmallmatrix} 0 & 1 \\ -1 & 0 \end{psmallmatrix}$ and $T = \begin{psmallmatrix} 1 & 0 \\ 0 & -1 \end{psmallmatrix}$. Then
		\[
		\Delta_A(T) = \begin{pmatrix} 0 & -2 \\ -2 & 0 \end{pmatrix}, \qquad
		\Delta_A^2(T) = \begin{pmatrix} -4 & 0 \\ 0 & 4 \end{pmatrix} = -4\,T,
		\]
		so $\omega = 2$. Here $\Delta_A(T)$ has eigenvalues $\pm 2$ and $T$ has eigenvalues $\pm 1$; neither is nilpotent.
	\end{example}
	
	\bigskip
	\section{Roots with Real Parts of the Same Sign}\label{sec:nonzero_real}
	\bigskip
	
	When $z_1, z_2$ have non-zero real parts of the same sign, $F(\lambda)$ decays to zero in one temporal direction and the spectral collapse is total. We treat the subcases by multiplicity and reality. The following observation is used repeatedly.
	
	\begin{remark}\label{rem:shift_trick}
		If $T$ satisfies~\eqref{eq:general_relation}, then so does $S := \Delta_A(T)$, since
		\[
		\Delta_A^2(S) + \alpha\,\Delta_A(S) + \beta\, S
		= \Delta_A\!\bigl(\Delta_A^2(T) + \alpha\,\Delta_A(T) + \beta\, T\bigr) = \Delta_A(0) = 0 .
		\]
		Consequently, any statement asserting that every solution of~\eqref{eq:general_relation} with a given pair of roots is nilpotent applies to $\Delta_A(T)$ as well, since one may run the argument a second time with $S$ in place of $T$. This is the case in every regime treated below, so the nilpotency of $\Delta_A(T)$ requires no separate proof and we do not repeat it.
	\end{remark}
	
	\subsection{Distinct real roots of the same sign ($z_1 z_2 > 0$)}\label{subsec:same_sign}
	
	Assume $z_1 > z_2$, both positive or both negative. Then  $$F(\lambda) = e^{z_1\lambda} C_1 + e^{z_2\lambda} C_2.$$
	
	\begin{theorem}\label{thm:same_sign}
		If $z_1, z_2 \in \mathbb{R}$ satisfy $z_1 z_2 > 0$ and $z_1 \neq z_2$, then $T$ and $\Delta_A(T)$ are nilpotent.
	\end{theorem}
	\begin{proof}
		If $z_1 > z_2 > 0$, then as $\lambda \to -\infty$ both exponentials tend to $0$, so
		\[
		\left\lVert F(\lambda)\right\rVert \leq e^{z_1\lambda}\, \left\lVert C_1\right\rVert + e^{z_2\lambda}\, \left\lVert C_2\right\rVert \to 0 .
		\]
		With Lemma~\ref{lem:spectral_tools}($iii$) and spectral invariance~\eqref{eq:spectral_invariance}, $$0 \le r(T) = r(F(\lambda)) \le \left\lVert F(\lambda)\right\rVert \to 0,$$ so $r(T) = 0$ and $T$ is nilpotent. If $0 > z_1 > z_2$, use $\lambda \to +\infty$. The argument applies to \emph{every} solution of~\eqref{eq:general_relation} with these roots; by Remark~\ref{rem:shift_trick}, $\Delta_A(T)$ is such a solution and is therefore nilpotent as well.
	\end{proof}
	
	\subsection{Repeated non-zero real root ($z_1 = z_2 = z_0 \neq 0$)}\label{subsec:repeated}
	
	Here $D = 0$, $\beta = \alpha^2/4 > 0$, and
	\begin{equation}\label{eq:repeated_solution}
		F(\lambda) = e^{z_0 \lambda}\bigl(T + \lambda\, C_2\bigr), \qquad C_2 = \Delta_A(T) - z_0\, T.
	\end{equation}
	
	\begin{theorem}\label{thm:repeated}
		If $z_0 \neq 0$, then $T$ and $\Delta_A(T)$ are nilpotent.
	\end{theorem}
	\begin{proof}
		Assume $z_0 > 0$ (the case $z_0 < 0$ is symmetric, via $\lambda \to +\infty$). Since $e^{z_0\lambda}|\lambda| \to 0$ as $\lambda \to -\infty$,
		\[
		\left\lVert F(\lambda)\right\rVert \leq e^{z_0\lambda}\bigl(\left\lVert T\right\rVert + |\lambda|\,\left\lVert C_2\right\rVert\bigr) \to 0 ,
		\]
		so $r(T) = 0$ and $T$ is nilpotent. This holds for every solution of~\eqref{eq:general_relation} with the repeated root $z_0$; by Remark~\ref{rem:shift_trick}, $\Delta_A(T)$ satisfies the same relation and is therefore nilpotent as well.
	\end{proof}
	
	\subsection{Complex conjugate roots with non-zero real part}\label{subsec:spiral}
	
	Let $z_{1,2} = \gamma \pm i\omega$ with $\gamma \neq 0$ and $\omega \neq 0$. Writing the general solution in real form, the orbit is the spiral
	\begin{equation}\label{eq:spiral}
		F(\lambda) = e^{\gamma\lambda}\bigl(\cos(\omega\lambda)\, K_1 + \sin(\omega\lambda)\, K_2\bigr),
	\end{equation}
	with $K_1 = T$ and $K_2 = \tfrac{1}{\omega}(\Delta_A(T) - \gamma T)$.
	
	\begin{theorem}\label{thm:spiral}
		If $\gamma \neq 0$, then $T$ and $\Delta_A(T)$ are nilpotent.
	\end{theorem}
	\begin{proof}
		Assume that $\gamma > 0$. The case $\gamma < 0$ is symmetric. Since $|\cos(\omega\lambda)| \leq 1$ and $|\sin(\omega\lambda)| \leq 1$, the factor $\cos(\omega\lambda)K_1 + \sin(\omega\lambda)K_2$ is bounded in norm by $\left\lVert K_1\right\rVert + \left\lVert K_2\right\rVert$, so
		\[
		\left\lVert F(\lambda)\right\rVert \leq e^{\gamma\lambda}\bigl(\left\lVert K_1\right\rVert + \left\lVert K_2\right\rVert\bigr) \longrightarrow 0
		\qquad \text{as } \lambda \to -\infty ,
		\]
		and therefore $r(T) = 0$, which means that $T$ is nilpotent. The argument applies to every solution of~\eqref{eq:general_relation} with roots $\gamma \pm i\omega$, so $\Delta_A(T)$ is nilpotent as well by Remark~\ref{rem:shift_trick}.
	\end{proof}

\begin{remark}
	In the present case the conclusion can also be reached without any norm estimate. At the points $\lambda_m = 2m\pi/\omega$, $m \in \mathbb{Z}$, the trigonometric factor returns to $K_1 = T$, so that
	\[
	F(\lambda_m) = e^{\gamma\lambda_m}\, T .
	\]
	Since $F(\lambda_m)$ is similar to $T$, the spectral radius is unchanged. Homogeneity then gives $r(T) = e^{\gamma\lambda_m}\, r(T)$ for every $m$. For $m \neq 0$ the factor $e^{\gamma\lambda_m}$ differs from $1$, from where it follows that $r(T) = 0$.
\end{remark}
	
	\bigskip
	\section{The Mixed Case: One Zero Root}\label{sec:mixed}
	\bigskip
	
	This is $\beta = 0$, $\alpha \neq 0$, with roots $z_1 = 0$, $z_2 = -\alpha$. The relation becomes $\Delta_A^2(T) + \alpha\, \Delta_A(T) = 0$, and $F'' + \alpha F' = 0$ integrates to
	\begin{equation}\label{eq:mixed_trajectory}
		F(\lambda) = \left(T + \frac{1}{\alpha}\,\Delta_A(T)\right) - \frac{e^{-\alpha\lambda}}{\alpha}\, \Delta_A(T).
	\end{equation}
	
	\begin{proposition}\label{prop:mixed}
		If $z_1 = 0$ and $z_2 \neq 0$, then $\Delta_A(T)$ is nilpotent.
	\end{proposition}
	\begin{proof}
			Since $F'(\lambda)$ is similar to $\Delta_A(T)$, we have that $$r(F'(\lambda)) = r(\Delta_A(T))$$ for every $\lambda$. On the other hand, differentiating~\eqref{eq:mixed_trajectory} gives $$F'(\lambda) = e^{-\alpha\lambda}\,\Delta_A(T),$$ and homogeneity of the spectral radius yields $$r(F'(\lambda)) = e^{-\alpha\lambda}\, r(\Delta_A(T)).$$ Comparing the two expressions,
		\[
		r(\Delta_A(T)) = e^{-\alpha\lambda}\, r(\Delta_A(T)) \qquad \text{for all } \lambda \in \mathbb{R} .
		\]
			Taking any $\lambda \neq 0$, we have $e^{-\alpha\lambda} \neq 1$ because $\alpha \neq 0$, and this forces $r(\Delta_A(T)) = 0$.
	\end{proof}
	
	\begin{example}\label{ex:mixed}
		In the setting of Proposition~\ref{prop:mixed} the matrix $T$ itself need not be nilpotent. Let
		\[
		A = \begin{pmatrix} 1 & 0 \\ 0 & 0 \end{pmatrix}, \qquad T = \begin{pmatrix} 1 & 0 \\ 1 & 1 \end{pmatrix} .
		\]
		Then
		\[
		\Delta_A(T) = \begin{pmatrix} 0 & 0 \\ -1 & 0 \end{pmatrix}, \qquad
		\Delta_A^2(T) = \begin{pmatrix} 0 & 0 \\ 1 & 0 \end{pmatrix} = -\,\Delta_A(T),
		\]
		so the relation~\eqref{eq:general_relation} holds with $\alpha = 1$ and $\beta = 0$, that is, with roots $z_1 = 0$ and $z_2 = -1$. Here $\Delta_A(T)$ is nilpotent, while $T$ is invertible with both eigenvalues equal to $1$.
	\end{example}
	
	\bigskip
	\section{Dimensional Resonance and the Hyperbolic Case}\label{sec:resonance}
	\bigskip
	
	In the hyperbolic regime $\beta < 0$ the roots are real and straddle the origin: $z_1 > 0 > z_2$. The orbit $F(\lambda) = e^{z_1\lambda} C_1 + e^{z_2\lambda} C_2$ diverges in both temporal directions, so the one-sided argument of Section~\ref{sec:nonzero_real} does not apply. Nilpotency is now governed by an arithmetic condition relating the exponents and the dimension.
	
	\subsection{The resonance dichotomy}
	
	\begin{theorem}\label{thm:resonance}
		Let $z_1 > 0 > z_2$ and suppose~\eqref{eq:general_relation} holds with these roots. Consider the condition
		\begin{equation}\label{eq:resonance}
			(\star)\qquad j\, z_1 + (k-j)\, z_2 \neq 0 \quad\text{for all } k \in \{1,\dots,n\},\ j \in \{1,\dots,k-1\}.
		\end{equation}
		\begin{enumerate}[\rm($i$)]
			\item If $(\star)$ holds, then every $T \in M_n(\mathbb{C})$ satisfying~\eqref{eq:general_relation} with roots $z_1,z_2$ is nilpotent, and so is $\Delta_A(T)$.
			\item If $(\star)$ fails, then there exist $A, T \in M_n(\mathbb{C})$ satisfying~\eqref{eq:general_relation} such that neither $T$ nor $\Delta_A(T)$ is nilpotent.
		\end{enumerate}
	\end{theorem}
	\begin{proof}
		($i$) Fix $k \le n$. By invariance, $\Tr(F(\lambda)^k) = \Tr(T^k)$ is constant in $\lambda$. Writing $F(\lambda) = e^{z_1\lambda}C_1 + e^{z_2\lambda}C_2$ and expanding the $k$-th power,
		\begin{equation}\label{eq:trace_expansion_detail}
			\Tr(F(\lambda)^k) = \sum_{j=0}^{k} e^{\mu_j \lambda}\, \Tr(W_j), \qquad \mu_j := j\, z_1 + (k-j)\, z_2,
		\end{equation}
		where $W_j$ is the sum of the ordered products of $C_1, C_2$ containing exactly $j$ factors $C_1$ and $k-j$ factors $C_2$ (each such product contributes the exponential $e^{\mu_j\lambda}$).
		The endpoints satisfy $\mu_0 = k z_2 \neq 0$ and $\mu_k = k z_1 \neq 0$, while the interior exponents $\mu_1,\dots,\mu_{k-1}$ are nonzero by $(\star)$. Moreover, the exponents $\mu_j$ are pairwise distinct, since consecutive differences equal $z_1 - z_2 > 0$. Therefore, the function
		\[
		\lambda \mapsto \sum_{j=0}^{k} e^{\mu_j \lambda}\, \Tr(W_j) - \Tr(T^k)\, e^{0 \cdot \lambda}
		\]
		vanishes identically on $\mathbb{R}$, and its exponents $\mu_0, \dots, \mu_k, 0$ are pairwise distinct. Lemma~\ref{lem:exp_independence} now gives $\Tr(W_j) = 0$ for all $j$, as well as $\Tr(T^k) = 0$. As $1 \le k \le n$ was arbitrary, Lemma~\ref{lem:newton} shows that $T$ is nilpotent. By Remark~\ref{rem:shift_trick}, $\Delta_A(T)$ is nilpotent too.
		
			($ii$) If $(\star)$ fails, then there are $k_0 \leq n$ and $j_0 \in \{1,\dots,k_0-1\}$ with
		\begin{equation}\label{eq:resonance_failure}
			j_0 z_1 + (k_0 - j_0)\, z_2 = 0 .
		\end{equation}
			In particular $z_1/z_2 = -(k_0-j_0)/j_0$ is rational, so we may write $z_1/z_2 = -p/q$ in lowest terms. Dividing~\eqref{eq:resonance_failure} by $z_2 \neq 0$ and then multiplying by $q$ gives $j_0 p = q(k_0 - j_0)$, that is,
		\[
		q\, k_0 = j_0\, (p+q) .
		\]
		Since $\gcd(p+q,q) = \gcd(p,q) = 1$, it follows that $p+q$ divides $k_0$, and therefore $p + q \leq k_0 \leq n$. 	Proposition~\ref{prop:construction} below now produces $A, T \in M_n(\mathbb{C})$ satisfying~\eqref{eq:general_relation} with $T$ and $\Delta_A(T)$ not nilpotent.
	\end{proof}
	
	\subsection{An explicit cyclic construction}
	
\begin{proposition}\label{prop:construction}
	Let $z_1 > 0 > z_2$ with $z_1/z_2 = -p/q$ in lowest terms, and set $m = p+q$. 	Then there exist $A, T \in M_m(\mathbb{C})$ with $T$ and $\Delta_A(T)$ invertible satisfying
	\[
	\Delta_A^2(T) + \alpha\, \Delta_A(T) + \beta\, T = 0, \qquad \alpha = -(z_1+z_2),\ \ \beta = z_1 z_2 .
	\]
		Consequently, for every $n \geq m$ the matrices $A \oplus 0_{\,n-m}$ and $T \oplus 0_{\,n-m}$ satisfy the same relation in $M_n(\mathbb{C})$, and neither $T \oplus 0_{\,n-m}$ nor its commutator with $A \oplus 0_{\,n-m}$ is nilpotent.
\end{proposition}
\begin{proof}
	Let $D = \operatorname{diag}(d_1,\dots,d_m)$ be diagonal and let $X \in M_m(\mathbb{C})$ have entries $X_{ij}$. Then $(DX)_{ij} = d_i X_{ij}$ and $(XD)_{ij} = d_j X_{ij}$, so that
	\[
	\Delta_D(X)_{ij} = (d_i - d_j)\, X_{ij}, \qquad 1 \leq i,j \leq m ,
	\]
	and, iterating, $\Delta_D^{\ell}(X)_{ij} = (d_i - d_j)^{\ell} X_{ij}$ for every $\ell \geq 0$. Consequently, writing $P(z) = z^2 + \alpha z + \beta$ for the characteristic polynomial,
	\[
	\bigl(\Delta_D^2(X) + \alpha\, \Delta_D(X) + \beta\, X\bigr)_{ij} = P(d_i - d_j)\, X_{ij} .
	\]
	Hence every matrix $X$ with $X_{ij} = 0$ whenever $d_i - d_j \notin \{z_1, z_2\}$ satisfies the relation, since each entry above vanishes either because $X_{ij} = 0$ or because $d_i - d_j$ is a root of $P$.
	
	\medskip 
	
	Now choose $\delta_1,\dots,\delta_m \in \{z_1, z_2\}$ with exactly $q$ of them equal to $z_1$ and $p$ equal to $z_2$. Since $z_1/z_2 = -p/q$ gives $q z_1 = -p z_2$, we obtain
	\[
	\sum_{i=1}^{m} \delta_i = q\, z_1 + p\, z_2 = 0 .
	\]
	Set $d_1 = 0$ and $d_{i+1} = d_i + \delta_i$ for $i = 1,\dots,m-1$, so that $d_{i+1} - d_i = \delta_i$ for these $i$. Since $d_m = \sum_{i=1}^{m-1} \delta_i = -\delta_m$, we also have $d_1 - d_m = \delta_m$. Thus all $m$ differences $d_2 - d_1, \dots, d_m - d_{m-1}, d_1 - d_m$ belong to $\{z_1, z_2\}$.
	
	Now let $A = \operatorname{diag}(d_1,\dots,d_m)$ with these $d_i$, and let
	\[
	T = \begin{pmatrix}
		0 & 0 & \cdots & 0 & 1 \\
		1 & 0 & \cdots & 0 & 0 \\
		0 & 1 & \cdots & 0 & 0 \\
		\vdots & \vdots & \ddots & \vdots & \vdots \\
		0 & 0 & \cdots & 1 & 0
	\end{pmatrix}.
	\]
		That is, the nonzero entries of $T$ are $T_{i+1,i} = 1$ for $1 \leq i \leq m-1$ and $T_{1,m} = 1$, and all remaining entries are zero. The differences attached to the nonzero positions are $\delta_1, \dots, \delta_m$, which lie in $\{z_1, z_2\}$ by the previous paragraph. Hence $T_{ij} = 0$ whenever $d_i - d_j \notin \{z_1, z_2\}$, and the first paragraph shows that
	\[
	\Delta_A^2(T) + \alpha\, \Delta_A(T) + \beta\, T = 0 .
	\]
	
	The matrix $T$ is a permutation matrix, hence invertible, and its eigenvalues are the $m$-th roots of unity. Moreover, $\Delta_A(T)_{i+1,i} = \delta_i$ for $1 \leq i \leq m-1$ and $\Delta_A(T)_{1,m} = \delta_m$, while all other entries vanish. Thus $\Delta_A(T)$ has the same pattern of nonzero positions as $T$, with the entries $\delta_1, \dots, \delta_m$ in place of the ones, 	so $\det \Delta_A(T) = (-1)^{m-1}\, \delta_1 \cdots \delta_m \neq 0$ and $\Delta_A(T)$ is invertible as well.
	
	Finally, let $A' = A \oplus 0_{\,n-m}$ and $T' = T \oplus 0_{\,n-m}$. The zero matrix satisfies the relation trivially, so $A'$ and $T'$ satisfy it blockwise, and $\Delta_{A'}(T') = \Delta_A(T) \oplus 0_{\,n-m}$. The nonzero eigenvalues of $T$ and of $\Delta_A(T)$ persist in these direct sums, so neither $T'$ nor $\Delta_{A'}(T')$ is nilpotent.
\end{proof}
	
	\begin{corollary}\label{cor:irrational}
		If $z_1/z_2 \notin \mathbb{Q}$, then $T$ and $\Delta_A(T)$ are nilpotent, for every dimension $n$.
	\end{corollary}
	\begin{proof}
		An irrational ratio admits no solution of $j z_1 + (k-j)z_2 = 0$ with $j \in \{1,\dots,k-1\}$, so $(\star)$ holds for all $n$. Now apply Theorem~\ref{thm:resonance}($i$).
	\end{proof}
	
	\begin{corollary}\label{cor:dimensional}
		Let $z_1 > 0 > z_2$ and write $z_1/z_2 = -p/q$ in lowest terms. Then:
		\begin{enumerate}[\rm($i$)]
			\item if $p + q > n$, then every $T \in M_n(\mathbb{C})$ satisfying~\eqref{eq:general_relation} with roots $z_1, z_2$ is nilpotent, and so is $\Delta_A(T)$;
			\item if $p + q \leq n$, then there exist $A, T \in M_n(\mathbb{C})$ satisfying~\eqref{eq:general_relation} such that neither $T$ nor $\Delta_A(T)$ is nilpotent.
		\end{enumerate}
	\end{corollary}
	\begin{proof}
		Fix $k \in \{1,\dots,n\}$ and let $j \in \{1,\dots,k-1\}$. As in the proof of Theorem~\ref{thm:resonance}($ii$), the equality $j z_1 + (k-j) z_2 = 0$ is equivalent to
		\[
		q\, k = j\, (p+q) .
		\]
		Suppose first that some $j \in \{1,\dots,k-1\}$ satisfies this equality. Then $p+q$ divides $q k$, and since $\gcd(p+q,q) = \gcd(p,q) = 1$, it divides $k$. Suppose conversely that $p+q$ divides $k$, say $k = (p+q)t$ with $t$ a positive integer, and put $j = qt$. Then $q k = q(p+q)t = j(p+q)$, so the equality holds. Moreover $j = qt \geq 1$ and $k - j = pt \geq 1$, so that $j$ lies in the required range. Therefore such a $j$ exists if and only if $p + q$ divides $k$.
		
		Consequently, $(\star)$ fails if and only if $p+q$ divides some $k \in \{1,\dots,n\}$, and this happens exactly when $p + q \leq n$. Indeed, if $p+q \leq n$ we may take $k = p+q$, while conversely every positive multiple of $p+q$ is at least $p+q$.
		
		Thus $(\star)$ holds when $p + q > n$, and ($i$) follows from Theorem~\ref{thm:resonance}($i$). When $p + q \leq n$ the condition $(\star)$ fails, and ($ii$) follows from Theorem~\ref{thm:resonance}($ii$).
	\end{proof}

	\subsection{Constructive counterexamples}
	
	\begin{example}\label{ex:opposite_2x2}
		Let
		\[
		A = \begin{pmatrix} 1 & 0 \\ 0 & -1 \end{pmatrix}, \qquad
		T = \begin{pmatrix} 0 & 1 \\ 1 & 0 \end{pmatrix}
		\]
		in $M_2(\mathbb{C})$. A direct computation gives
		\[
		\Delta_A(T) = \begin{pmatrix} 0 & 2 \\ -2 & 0 \end{pmatrix}, \qquad
		\Delta_A^2(T) = \begin{pmatrix} 0 & 4 \\ 4 & 0 \end{pmatrix} = 4\, T ,
		\]
		so that $\Delta_A^2(T) - 4T = 0$, with roots $z_1 = 2$ and $z_2 = -2$. Here $p = q = 1$ and $p + q = 2 = n$, so $(\star)$ fails at $k = 2$. Both $T$ and $\Delta_A(T)$ are invertible, with $\det T = -1$ and $\det \Delta_A(T) = 4$, and in particular neither is nilpotent. This is the case $m = 2$ of Proposition~\ref{prop:construction}, with $\delta_1 = z_1$ and $\delta_2 = z_2$.
	\end{example}
	
	\begin{example}\label{ex:3x3_resonance}
		Let
		\[
		A = \begin{pmatrix} 1 & 0 & 0 \\ 0 & -1 & 0 \\ 0 & 0 & 0 \end{pmatrix}, \qquad
		T = \begin{pmatrix} 0 & 1 & 0 \\ 0 & 0 & 1 \\ 1 & 0 & 0 \end{pmatrix}
		\]
		in $M_3(\mathbb{C})$. A direct computation gives
		\[
		\Delta_A(T) = \begin{pmatrix} 0 & 2 & 0 \\ 0 & 0 & -1 \\ -1 & 0 & 0 \end{pmatrix}, \qquad
		\Delta_A^2(T) = \begin{pmatrix} 0 & 4 & 0 \\ 0 & 0 & 1 \\ 1 & 0 & 0 \end{pmatrix} ,
		\]
		so that $$\Delta_A^2(T) - \Delta_A(T) - 2T = 0,$$ with roots $z_1 = 2$ and $z_2 = -1$. Here $p = 2$, $q = 1$ and $p + q = 3 = n$, the resonance occurring at $k = 3$. Both $T$ and $\Delta_A(T)$ are invertible, with $\det T = 1$ and $\det \Delta_A(T) = 2$, and the eigenvalues of $T$ are the cube roots of unity. This is the case $m = 3$ of Proposition~\ref{prop:construction}, with one difference equal to $z_1$ and two equal to $z_2$.
	\end{example}
	
	\bigskip
	\section{Higher-Order Commutator Relations}\label{sec:higher}
	\bigskip
	
	The one-sided squeezing argument is not tied to the order of the relation, and we isolate it here as a general half-plane principle. From this viewpoint, Sections~\ref{sec:imaginary}--\ref{sec:resonance} treat in full the boundary cases of this principle at order two and with real coefficients, in which the roots meet or straddle the imaginary axis.
	
	\begin{theorem}\label{thm:higher}
		Let $m \geq 1$ and let $c_0,\dots,c_m \in \mathbb{C}$ with $c_m \neq 0$. Suppose that
		\begin{equation}\label{eq:higher_relation}
			\sum_{k=0}^{m} c_k\, \Delta_A^k(T) = 0
		\end{equation}
		and that there exists $\theta \in \mathbb{R}$ such that
		\[
		\Rea\bigl(e^{-i\theta} z\bigr) > 0 \qquad \text{for every root } z \text{ of } P(z) = \sum_{k=0}^{m} c_k z^k .
		\]
		Then $\Delta_A^j(T)$ is nilpotent for every $j \geq 0$. In particular, both $T$ and $\Delta_A(T)$ are nilpotent.
	\end{theorem}
	\begin{proof}
		We first treat the case $\theta = 0$, so that every root of $P$ has positive real part. By Lemma~\ref{lem:derivatives}, the orbit $F(\lambda) = e^{\lambda A} T e^{-\lambda A}$ satisfies
		\[
		\sum_{k=0}^{m} c_k\, F^{(k)}(\lambda) = e^{\lambda A}\Bigl(\sum_{k=0}^{m} c_k\, \Delta_A^k(T)\Bigr) e^{-\lambda A} = 0 ,
		\]
		which is a linear differential equation with constant coefficients and characteristic polynomial $P$. Hence
		\[
		F(\lambda) = \sum_{i} P_i(\lambda)\, e^{z_i \lambda} ,
		\]
		where $z_1, \dots, z_s$ are the distinct roots of $P$ and each $P_i$ is a matrix polynomial of degree less than the multiplicity of $z_i$.
		
		Let $\lambda \to -\infty$. For each $i$ we have $\lvert e^{z_i \lambda}\rvert = e^{\Rea(z_i)\lambda} \to 0$, and this exponential decay dominates the polynomial growth of $P_i(\lambda)$, so every term tends to $0$ and $\left\lVert F(\lambda)\right\rVert \to 0$. Combining spectral invariance~\eqref{eq:spectral_invariance} with Lemma~\ref{lem:spectral_tools}($iii$),
		\[
		0 \leq r(T) = r(F(\lambda)) \leq \left\lVert F(\lambda)\right\rVert \longrightarrow 0 ,
		\]
		so $r(T) = 0$ and $T$ is nilpotent.
		
		For a general $\theta$, put $B = e^{-i\theta} A$. Then $\Delta_B = e^{-i\theta}\, \Delta_A$, hence $\Delta_A^k(T) = e^{ik\theta}\, \Delta_B^k(T)$, and~\eqref{eq:higher_relation} may be rewritten as
	\[
	\sum_{k=0}^{m} c_k\, e^{ik\theta}\, \Delta_B^k(T) = \sum_{k=0}^{m} c_k\, \Delta_A^k(T) = 0 .
	\]
	This is a relation of the same form for $B$, with associated polynomial $$Q(w) = \sum_{k} c_k e^{ik\theta} w^k = P(e^{i\theta} w).$$ Its roots are the numbers $e^{-i\theta} z$ with $z$ a root of $P$, and these have positive real part by hypothesis. The case already treated, applied to $B$ in place of $A$, shows that $T$ is nilpotent.
		
		It remains to pass to the iterated commutators. Put $S = \Delta_A(T)$. Applying $\Delta_A$ to~\eqref{eq:higher_relation} gives
		\[
		\sum_{k=0}^{m} c_k\, \Delta_A^k(S) = \Delta_A\Bigl(\sum_{k=0}^{m} c_k\, \Delta_A^k(T)\Bigr) = \Delta_A(0) = 0 ,
		\]
		so $S$ satisfies the same relation, with the same polynomial $P$, and is therefore nilpotent by what we have just proved. Iterating this argument yields the nilpotency of $\Delta_A^j(T)$ for every $j \geq 0$.
	\end{proof}
	
	\begin{remark}\label{rem:higher_sharp}
		Geometrically, the hypothesis says that the roots of $P$ lie in an open half-plane whose boundary passes through the origin. The choices $\theta = 0$ and $\theta = \pi$ recover the two configurations in which all roots have real part of the same sign.
	\end{remark}
	
	\begin{remark}\label{rem:higher_sharp2}
			The hypothesis cannot be dropped. If the roots of $P$ do not lie in a common open half-plane whose boundary passes through the origin, the collapse can fail, and Examples~\ref{ex:oscillatory},~\ref{ex:opposite_2x2} and~\ref{ex:3x3_resonance} exhibit non-nilpotent solutions of this kind. In each of them the roots lie on a single line through the origin and on opposite sides of it, in the first case on the imaginary axis and in the other two on the real axis.
	\end{remark}
	
	\bigskip
	\section{Classification of Nilpotency Regimes}\label{sec:classification}
	\bigskip
	
	We collect the second-order analysis into one statement.
	
	\begin{theorem}\label{thm:main}
		Let $A, T \in M_n(\mathbb{C})$ and $\alpha, \beta \in \mathbb{R}$ satisfy
		\[
		\Delta_A^2(T) + \alpha\, \Delta_A(T) + \beta\, T = 0,
		\]
		and let $z_1, z_2$ be the roots of $P(z) = z^2 + \alpha z + \beta$.
		\begin{enumerate}[\rm($i$)]
			\item \emph{Double zero root} ($\alpha = \beta = 0$): $\Delta_A(T)$ is nilpotent, while $T$ need not be, as any non-nilpotent $T$ commuting with $A$ shows.
			\item \emph{Roots with non-zero real parts of the same sign} (distinct real roots with $z_1 z_2 > 0$, a repeated root $z_0 \neq 0$, or complex conjugate roots $\gamma \pm i\omega$ with $\gamma \neq 0$): both $T$ and $\Delta_A(T)$ are nilpotent.
			\item \emph{Purely imaginary roots} ($z_{1,2} = \pm i\omega$ with $\omega \neq 0$): neither matrix need be nilpotent.
			\item \emph{One zero root and one non-zero real root} ($\beta = 0$, $\alpha \neq 0$): $\Delta_A(T)$ is nilpotent, while $T$ need not be.
				\item \emph{Opposite-sign real roots} ($\beta < 0$): if $z_1/z_2 \notin \mathbb{Q}$, or if $z_1/z_2 = -p/q$ in lowest terms with $p + q > n$, then both $T$ and $\Delta_A(T)$ are nilpotent. Otherwise there are $A$ and $T$ satisfying the relation such that neither $T$ nor $\Delta_A(T)$ is nilpotent.
		\end{enumerate}
	\end{theorem}
	
	\subsection{Summary of the regimes}

The behaviour of the orbit in each regime, and the resulting nilpotency conclusion, are collected in Table~\ref{tab:dynamical}. The location of the roots in terms of the parameters $\alpha$ and $\beta$ is shown in Figure~\ref{fig:spectral_map}.

\begin{table}[h!]
	\centering
	\footnotesize
	\renewcommand{\arraystretch}{1.4}
	\setlength{\tabcolsep}{4pt}
	\begin{tabular}{|l|l|l|}
		\hline
		\textbf{Roots of $P$} & \textbf{Behaviour of $F$} & \textbf{Necessarily nilpotent} \\ \hline
		$z_1 = z_2 = 0$ & linear growth & $\Delta_A(T)$ \\ \hline
		$\Rea z_1, \Rea z_2 \neq 0$, same sign & decays to $0$ in one direction & $T$ and $\Delta_A(T)$ \\ \hline
		$z_{1,2} = \pm i\omega$, $\omega \neq 0$ & bounded and periodic & neither \\ \hline
		$z_1 = 0$, $z_2 \neq 0$ real & tends to a constant in one direction & $\Delta_A(T)$ \\ \hline
		$z_1 > 0 > z_2$ & unbounded in both directions & conditional \\ \hline
	\end{tabular}
	\vspace{0.2cm}
	\caption{Behaviour of the orbit in each regime and the resulting nilpotency conclusion. In the last line the conclusion depends on the arithmetic condition $(\star)$ of Theorem~\ref{thm:resonance}.}
	\label{tab:dynamical}
\end{table}

\begin{figure}[h!]
	\centering
	\begin{tikzpicture}[scale=1.9]
		\begin{axis}[
			axis lines = middle,
			axis on top = true,
			xlabel = {$\alpha$},
			ylabel = {$\beta$},
			xmin = -3.6, xmax = 3.6,
			ymin = -2.4, ymax = 3.4,
			samples = 100,
			xtick = {0}, ytick = {0},
			xticklabels = {}, yticklabels = {},
			enlargelimits = false,
			clip = true,
			]
			\addplot [name path=parab, draw=none, domain=-3.6:3.6] {x^2/4};
			\addplot [name path=top,   draw=none, domain=-3.6:3.6] {3.4};
			\addplot [name path=zero,  draw=none, domain=-3.6:3.6] {0};
			\addplot [name path=bot,   draw=none, domain=-3.6:3.6] {-2.4};
			\addplot [blue!8]   fill between [of=parab and top];
			\addplot [green!12] fill between [of=zero and parab];
			\addplot [red!10]   fill between [of=bot and zero];
			\addplot [thick, dashed, domain=-3.6:3.6] {x^2/4};
			\draw [thick] (axis cs:-3.6,0) -- (axis cs:3.6,0);
			\node [align=center, font=\footnotesize] at (axis cs:0,2.75)
			{complex conjugate roots, $\Rea z_i \neq 0$\\ $T$ and $\Delta_A(T)$ nilpotent};
			\node [align=center, font=\scriptsize] at (axis cs:-2.55,0.45)
			{real roots,\\ same sign\\ $T$, $\Delta_A(T)$ nilpotent};
			\node [align=center, font=\scriptsize] at (axis cs:2.55,0.45)
			{real roots,\\ same sign\\ $T$, $\Delta_A(T)$ nilpotent};
			\node [align=center, font=\footnotesize] at (axis cs:0,-1.5)
			{real roots, opposite signs\\ nilpotency conditional on $(\star)$};
			\node [font=\scriptsize, rotate=0, anchor=west] at (axis cs:-2.9,2.15) {$D=0$};
			\node [anchor=west, font=\scriptsize] at (axis cs:-0.03,1.7)
			{$\alpha=0$: roots $\pm i\omega$, no nilpotency};
						\node [anchor=north west, font=\scriptsize] at (axis cs:-1.1,0.5)
			{K--S point: $\Delta_A(T)$ nilpotent};
			\node [anchor=north, font=\scriptsize] at (axis cs:-2.4,-0.05)
			{$\beta=0$: $\Delta_A(T)$ nilpotent};
			\filldraw [black] (axis cs:0,0) circle (1.6pt);
		\end{axis}
	\end{tikzpicture}
		\caption{The nilpotency regimes in the parameter plane. The dashed parabola $\beta=\alpha^2/4$ is the discriminant curve $D=0$; below it the roots of $P(z)=z^2+\alpha z+\beta$ are real, above it they are complex conjugates. Both $T$ and $\Delta_A(T)$ are nilpotent in the two regions with $\beta>0$, apart from the positive $\beta$-axis, where the roots are purely imaginary. On the axis $\beta=0$ only $\Delta_A(T)$ is nilpotent, and for $\beta<0$ the conclusion depends on the arithmetic condition $(\star)$.}\label{fig:spectral_map}
\end{figure}
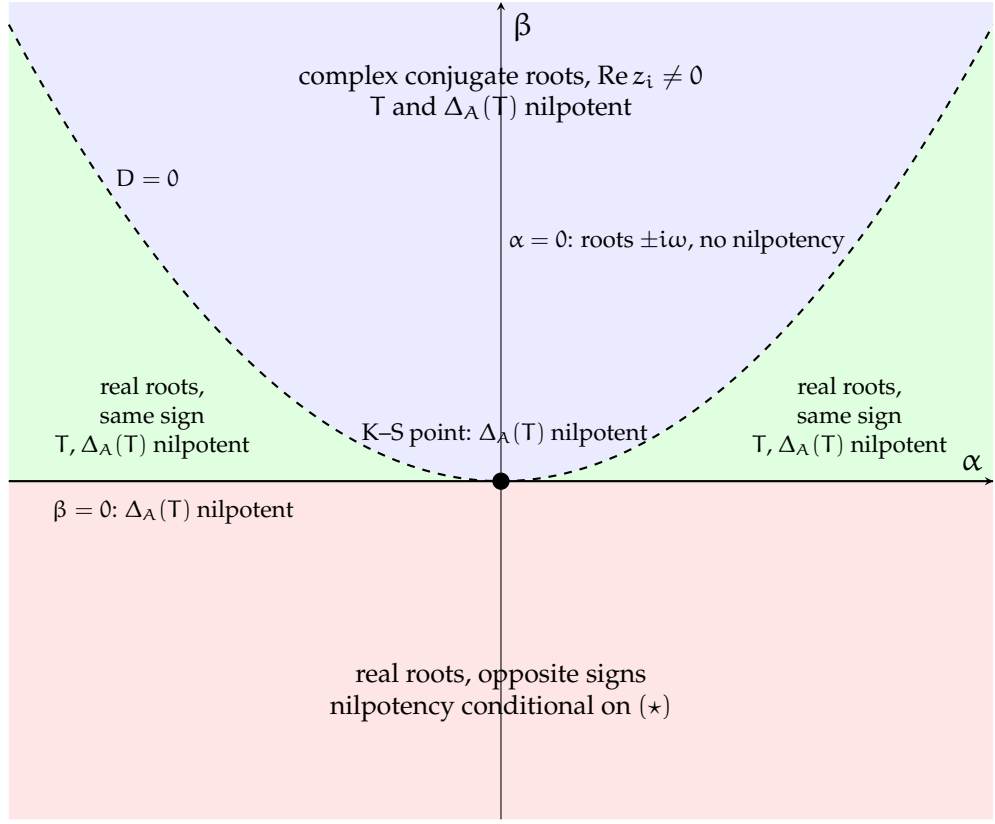

\subsection{The squeezing mechanism}

Whether the collapse reaches $T$ or only $\Delta_A(T)$ is decided by whether the roots of $P$ have non-zero real part.
\begin{enumerate}[\rm($i$)]
	\item \emph{Total collapse.} If both roots have non-zero real parts of the same sign, then $\left\lVert F(\lambda)\right\rVert \to 0$ in one of the two directions. The spectral radius is constant along the orbit, so it must be zero, and both $T$ and $\Delta_A(T)$ are nilpotent.
	\item \emph{Partial collapse.} If one root is zero, the orbit need not decay and $T$ may remain non-nilpotent. Its derivative $F'$ is nevertheless governed by the remaining root alone. For $z_1 = 0$ and $z_2 \neq 0$ the derivative decays exponentially, while for the double zero root the orbit grows linearly in $\lambda$ although its spectral radius stays constant. In both cases $\Delta_A(T)$ is nilpotent.
	\item \emph{No collapse.} If the roots are purely imaginary the orbit is bounded and no conclusion follows. If they are real of opposite signs the orbit is unbounded in both directions, so the one-sided decay argument is unavailable. The traces $\Tr(T^k)$ take over in this case, and a non-zero spectrum survives exactly when $(\star)$ fails.
\end{enumerate}
	
	\bigskip
	\section{Concluding Remarks}\label{sec:conclusion}
	\bigskip
	
	We have described the spectral rigidity of commutators under second-order linear relations through the asymptotic dynamics of the similarity orbit. The mechanism, which we have called \emph{spectral squeezing}, consists in the orbit following the trajectories prescribed by the characteristic exponents while carrying a constant spectrum. The Kleinecke--Shirokov phenomenon appears as the zero-root endpoint of this mechanism. One-sided exponential bounds force global nilpotency, while hyperbolic configurations admit non-nilpotent solutions exactly when the resonance condition $(\star)$ fails (Theorem~\ref{thm:resonance}, Proposition~\ref{prop:construction}).
	
	Theorem~\ref{thm:higher} shows that the half-plane principle reaches arbitrary order: if the roots of $\sum_{k=0}^{m} c_k z^k$ lie in a single open half-plane through the origin, then every iterated commutator $\Delta_A^j(T)$ is nilpotent. Two directions remain open. The first is a higher-order analogue of the resonance dichotomy, which would classify the configurations of roots that allow non-nilpotent solutions, together with the corresponding constructions.  		The second is the transport of the collapse arguments to a  complex Banach algebra $\mathcal{A}$, with nilpotency replaced by quasi-nilpotency. The arguments based on one-sided decay of the orbit carry over, since they use only the invariance of the spectral radius under similarity together with the bound $r(x) \leq \left\lVert x \right\rVert$, valid for every $x \in \mathcal{A}$. What is lost is the continuity of the spectral radius, which is upper semicontinuous but not continuous in general, and the proof of Theorem~\ref{thm:KS} given above depends on it.

	\bigskip
	\section*{Declarations}
	\noindent{\bf{Funding}}\\
	This work has been supported by the Ministry of Science, Technological Development and Innovation of the Republic of Serbia [Grant Number: 451-03-34/2026-03/200102].
	
	\vspace{0.5cm}
	
	\noindent{\bf{Availability of data and materials}}\\
	\noindent No data were used to support this study.
	\vspace{0.5cm}\\
	\noindent{\bf{Competing interests}}\\
	\noindent The author declares that he has no competing interests.
	\vspace{0.5cm}\\
	\noindent{\bf{Declaration of generative AI and AI-assisted technologies in the manuscript preparation process}}\\
	\noindent During the preparation of this work, the author used Claude (Anthropic) to verify the mathematical arguments, to assist in the derivation and drafting of parts of the results, and to improve the exposition; ChatGPT (OpenAI) and Gemini (Google) were additionally consulted for pre-submission checks. The author reviewed, verified and edited all generated content and takes full responsibility for the content of the published article.
	
	
	\vspace{0.5cm} 
	

\end{document}